\documentclass[aos, noinfoline]{imsart}         
\usepackage{amsmath,amsfonts,amsthm,amssymb}    
\usepackage[pageanchor]{hyperref}               
\usepackage{bbm}                                
\usepackage{mathptmx}                           
\usepackage[mathcal]{euscript}                  
\usepackage{enumitem}                           
\usepackage{mathtools}                          
\usepackage[margin=1.5in]{geometry}             
\usepackage{subcaption}
\usepackage{tikz}
\usetikzlibrary{patterns,shapes}

\usepackage{datetime}                            
\startlocaldefs
\def\journal@id{~}
\def\journal@name{~}
\def\journal@url{~}
\endlocaldefs

\newtheorem{theorem}{Theorem}[section]

\newtheorem{lemma}[theorem]{Lemma}
\newtheorem{proposition}[theorem]{Proposition}

\theoremstyle{definition}
\newtheorem{definition}[theorem]{Definition}
\theoremstyle{remark}
\newtheorem{remark}[theorem]{Remark}
\newtheorem{example}[theorem]{Example} 
\numberwithin{equation}{section}

\def\R{{\mathbb R}}
\def\N{{\mathbb N}}

\def\1{{\mathbbm 1}}
\def\E{{\mathbb E}}
\def\P{{\mathbb P}}
\def\ra{\rightarrow}

\renewcommand{\|}{\Vert}

\def\sphn{\mathbb{S}^{n-1}}
\def\seq{\mathbb{S}}
\def\onen{\iota^{(n)}}

\def\thn{\theta^{(n)}}

\def\ir{\mathbb{I}}

\begin{document}

\begin{frontmatter}

\title{Cram\'er's theorem is atypical}
\runtitle{Cram\'er's theorem is atypical}

\begin{aug}
  \author{\fnms{Nina}  \snm{Gantert} \thanksref{t1} \ead[label=e1]{gantert@ma.tum.de}},
  \author{\fnms{Steven Soojin} \snm{Kim} \thanksref{t2,t4} \ead[label=e2]{steven\_kim@brown.edu}},
  \and
  \author{\fnms{Kavita}  \snm{Ramanan} \thanksref{t1,t3,t4} \ead[label=e3]{kavita\_ramanan@brown.edu}}

  \thankstext{t1}{NG and KR would like to thank ICERM, Providence, for an invitation to the program ``Computational Challenges in Probability", where some of this work was initiated.}
  \thankstext{t2}{SSK was partially supported by a Department of Defense NDSEG fellowship.}  
  \thankstext{t3}{KR was partially supported by ARO grant W911NF-12-1-0222 and NSF grant DMS 1407504.}
  \thankstext{t4}{SSK and KR would also like to thank Microsoft Research New England for their hospitality during the Fall of 2014, when some of this work was completed.}

  \runauthor{Nina Gantert, Steven Soojin Kim, and Kavita Ramanan}

  \affiliation{Technische Universit\"at M\"unchen, Brown University, and Brown University}

  \address{Fakult\"at f\"ur Mathematik,\\
           Technische Universit\"at M\"unchen\\
          \printead{e1}}

  \address{Division of Applied Mathematics,\\
          Brown University\\
          \printead{e2,e3}}
\end{aug}

\begin{abstract} 
The empirical mean of $n$ independent and identically distributed
(i.i.d.) random variables $(X_1,\dots,X_n)$ can be viewed as a suitably normalized scalar projection of the $n$-dimensional random vector $X^{(n)}\doteq(X_1,\dots,X_n)$ in the direction of the unit
vector $n^{-1/2}(1,1,\dots,1) \in \sphn$. The large deviation
principle (LDP) for such projections as $n\rightarrow\infty$ is given
by the classical Cram\'er's theorem. We prove an LDP for
the 
sequence of normalized scalar projections of $X^{(n)}$ in the direction of a generic unit vector $\thn \in \sphn$, as $n\ra\infty$. This
LDP holds under fairly general conditions on the distribution of $X_1$,
and for ``almost every" sequence of directions $(\thn)_{n\in\N}$. The
associated rate function is ``universal'' in the sense 
that it  does not depend on the particular sequence of directions. 
Moreover, under mild additional conditions on the law of $X_1$, we
show that the universal rate function differs from the Cram\'er rate
function, thus showing that the sequence of directions
$n^{-1/2}(1,1,\dots,1) \in \sphn,$ $n \in \mathbb{N}$, 
corresponding to Cram\'er's theorem is atypical.
\end{abstract}

\begin{keyword}[class=MSC]
\kwd[Primary ]{60F10} 
\kwd[; secondary ]{60D05.} 
\end{keyword}

\begin{keyword}
\kwd{large deviations}
\kwd{projections}
\kwd{high-dimensional product measures}
\kwd{Cram\'er's theorem}
\kwd{rate function.}
\end{keyword}

\end{frontmatter}

\section{Introduction}\label{sec-intro}

Let $X^{(n)} = (X_1,\dots, X_n)$ be a sequence of $n$ independent and identically distributed (i.i.d.) $\R$-valued random variables with common distribution $\gamma \in \mathcal{P}(\R)$. A fundamental probabilistic question is how the empirical mean of $X^{(n)}$ behaves as the length of the sequence $n$ increases. From a geometric perspective, the empirical mean is a suitably normalized version of (the scalar component of) the projection of the $n$-dimensional vector $X^{(n)}$ in the direction of the unit vector $\onen$, defined by
\begin{equation}\label{onedefn}
  \onen \doteq \tfrac{1}{\sqrt{n}}(\overbracket{1,1,\dots, 1}^{n \text{ times }}) \in \sphn.
\end{equation}
In other words, we can write 
\begin{equation}
  W_\onen \doteq \frac{1}{\sqrt{n}} \langle X^{(n)}, \onen \rangle_n = \frac{1}{n}\sum_{i=1}^n X_i \, , \label{meanproj}
\end{equation}
where $\langle \cdot, \cdot\rangle_n$ denotes the Euclidean inner product. With some abuse of terminology, for $x\in \R^n$ and $v\in \sphn$, we hereby write the ``\emph{projection} of $x$ in the direction $v$" to refer to the scalar component $\langle x,v\rangle_n \in \R$ (rather than the vector $\langle x,v\rangle_nv\in \R^n$). Then, the expression \eqref{meanproj} indicates that questions on the empirical mean for large $n$ can be rephrased in a geometric language as questions on suitably normalized projections of high-dimensional random vectors. 

The classical Cram\'er's theorem characterizes the large deviations behavior of \eqref{meanproj}, the empirical mean of i.i.d.\ random variables, as $n\rightarrow\infty$. In particular, if $X_1\sim \gamma$ has some finite exponential moments, in the sense that 
\begin{equation}\label{weakexpon}
\exists\, t_0>0 \text{ s.t. } \forall\, |t|< t_0, \quad  \Lambda(t) \doteq \log \E[e^{tX_1}] < \infty,
\end{equation}
then we have the limit
\begin{equation*}
  \lim_{n\ra\infty} \frac{1}{n}\log \P(W_\onen \ge x) = -\Lambda^*(x),
\end{equation*}
where $^*$ denotes the \emph{Legendre transform}, 
\begin{equation}\label{legendef}
\Lambda^*(x) \doteq \sup_{t\in \R} \{ tx - \Lambda(t)\}. 
\end{equation}
We refer to \cite[\S 12]{rockafellar1970convex} for a review of the Legendre transform (also known as the \emph{convex conjugate}).

Given the geometric view of empirical means given by \eqref{meanproj}, it is natural to investigate analogs of Cram\'er's theorem for normalized projections in directions $\thn\in\sphn$ other than $\onen$. Such projections correspond to \emph{weighted} means,
\begin{equation}
W_{\theta}^{(n)}\doteq \frac{1}{\sqrt{n}} \langle X^{(n)}, \thn \rangle_n =  \frac{1}{n}\sum_{i=1}^n X_i \sqrt{n} \thn_i . \label{weightproj}
\end{equation}
Our main result is an LDP for $(W_{\theta}^{(n)})_{n\in\N}$ for almost every (in a sense that is specified below) sequence of directions $\theta = (\theta^{(1)},\theta^{(2)},\dots)$. In particular, we show that the associated rate function does not depend on $\theta$, and that it differs from the Cram\'er rate function $\Lambda^*$. That is, the sequence of directions $(\onen)_{n\in\N}$ corresponding to Cram\'er's theorem is ``atypical"!

\begin{remark}
While the LDP for \eqref{weightproj} is novel, the corresponding law of large numbers (LLN) and central limit theorem (CLT) for weighted sums are well known. For example, a weak LLN follows from Chebyshev's inequality, and a CLT follows from the Lindeberg conditions (see, e.g., \cite[\S VIII.4, Theorem 3]{feller1970probability2}).
\end{remark}

The outline of this note is as follows. In Section \ref{sec-main}, we
state our main results and discuss their relation to prior work. In Section
\ref{sec-sur}, we prove the claimed LDP. In Section \ref{sec-atyp}, we
establish that Cram\'er's theorem is atypical, and also comment on 
a generalization that is considered in \cite{gkr2}.

\section{Main results}\label{sec-main}

We first set some notation. Suppose the random variables $X_1,X_2,\dots$ are all defined on a common probability space $(\Omega,\mathcal{F},\P)$. Let $\|\cdot\|_{n}$ denote the Euclidean norm on $\R^n$. Write  $\sigma_{n-1}$ for the unique rotation invariant probability measure on $\sphn$, the unit sphere in $\R^n$. Let $\seq \doteq \prod_{n \in \N} \sphn$, and let $\pi_n:\seq\ra \sphn$ be the coordinate map such that for $\theta = (\theta^{(1)},\theta^{(2)},\dots)\in \seq$, we have $\pi_n(\theta) = \thn$. Let $\sigma$ be a probability measure on (the Borel sets of) $\seq$ such that
\begin{equation}\label{sigproj}\tag{H1}
\sigma \circ \pi_n^{-1} = \sigma_{n-1}, \quad n\in \N.
\end{equation}
The generic example to keep in mind that satisfies \eqref{sigproj}  is the product measure $\sigma=\bigotimes_{n\in \N} \sigma_{n-1}$, in which case the projection directions $\thn$, $n\in\N$, are independent under $\sigma$. However, our results allow for more general dependencies; for more discussion on $\sigma$ and the condition \eqref{sigproj}, see Remark \ref{rmk-sigma}.

For $\sigma$-a.e.\ $\theta\in\sphn$, we prove a large deviation principle for the sequence $(W_\theta^{(n)})_{n\in \N}$ with a rate function that does not depend on $\theta$. We refer to \cite{DemZeibook} for general background on large deviations. In particular, recall the following definition:

\begin{definition}
The sequence of probability measures $(\mu_n)_{n\in \N}\subset\mathcal{P}(\R)$ is said to satisfy a \emph{large deviation principle (LDP)} with a \emph{rate function} $\ir:\R\ra[0,\infty]$ if $\ir$ is lower semicontinuous, and for all Borel measurable sets $\Gamma\subset \R$, 
\begin{equation*}
  -\inf_{x\in \Gamma^\circ} \ir(x) \le  \liminf_{n\ra\infty} \tfrac{1}{n} \log \mu_n(\Gamma^\circ) \le \limsup_{n\ra\infty} \tfrac{1}{n} \log \mu_n(\bar \Gamma) \le -\inf_{x\in \bar{\Gamma}} \ir(x),
\end{equation*}
where $\Gamma^\circ$ and $\bar\Gamma$ denote the interior and closure of $\Gamma$, respectively. Furthermore, $\ir$ is said to be a \emph{good rate function} if it has compact level sets.

We say the sequence of $\R$-valued random variables $(\xi_n)_{n\in \N}$ satisfies an LDP if the sequence of laws $(\mu_n)_{n\in\N}$ given by $\mu_n = \P \circ \xi_n^{-1}$ satisfies an LDP.
\end{definition}

In particular, for empirical means of i.i.d.\ random variables, we recall the following classical result, due to \cite{cramer38sur, chernoff52measure}.

\begin{theorem}[Cram\'er]
Let $(X_n)_{n\in\N}$ be an i.i.d.\ sequence such that \eqref{weakexpon} holds, and let $\iota=(\iota^{(1)},\iota^{(2)},\dots)$ be defined as in \eqref{onedefn}. Then the sequence $(W_\onen)_{n\in \N}$ of \eqref{meanproj} satisfies an LDP with the good rate function $\ir_\iota$, given by
\begin{equation}\label{iriodef}
  \ir_\iota(w) \doteq \Lambda^*(w) = \sup_{t\in\R} \{ tw - \Lambda(t)\}.
\end{equation}
\end{theorem}

Let $\nu\in\mathcal{P}(\R)$ denote the standard one-dimensional Gaussian measure. In the sequel, we assume the following condition on $\Lambda$, the logarithmic moment generating function (log mgf) of $X_1\sim \gamma$:
\begin{equation}\label{strongexpon}
\forall \,t\in \R, \quad \int_\R |\Lambda(tu)|^4\nu(du) < \infty. \tag{H2}
\end{equation}
Note that \eqref{strongexpon} is stronger than even requiring the exponential moment condition in \eqref{weakexpon} to hold with $t_0=\infty$.  For an absolutely continuous $\gamma$ with density, a sufficient condition for \eqref{strongexpon} is that the decay of the tail of the density is strictly faster than exponential, in the following sense:

\begin{lemma}  \label{lem-tail}
Suppose that $\gamma$ has density $f$, and that there exist $p\in(1,\infty)$ and constants $0 < C_1, C_2, C_3 < \infty$ such that for $|x| > C_1$, we have
\begin{equation*}
  f(x) \le C_2 e^{-C_3|x|^p}.
\end{equation*}
Then there exists some constant $C< \infty$ such that $\Lambda$, the log mgf of $\gamma$, satisfies the following upper bound for all $t\in \R$:
\begin{equation*}
  \Lambda(t) \le C|t|^{p/(p-1)} +C.
\end{equation*}
Moreover, this implies that $\Lambda$ satisfies the condition \eqref{strongexpon}.
\end{lemma}
\begin{proof}
By Young's inequality applied to  the conjugate exponents $p$ and $\tfrac{p}{p-1}$, for $\epsilon > 0$ and $t,y\in \R$,
\begin{equation*}
  ty \le \left( \epsilon^{-1/p}|t|\right) \left(\epsilon^{1/p} |y|\right) \le \tfrac{p-1}{p} \epsilon^{-1/(p-1)}|t|^{p/(p-1)}  + \frac{\epsilon |y|^p}{ p}.
\end{equation*}
In the following, let $C$ absorb all constants, and let $0 < \epsilon < C_3p$ to find that for $t\in \R$,
\begin{align*}
\Lambda(t) &= \log \int_{|y|\le C_1} e^{ty} f(y) dy + \log \int_{|y| > C_1} e^{ty} f(y) dy\\
   &\le C_1|t| + \log C_2 + \log \int_{\R} e^{ty} e^{-C_3 |y|^p} dy\\
   &\le C_1|t|^{p/(p-1)} + C_1 + \tfrac{p-1}{p} \epsilon^{-1/(p-1)}|t|^{p/(p-1)}  -\tfrac{1}{p}\log(C_3p-\epsilon) + 1\\
   &= C |t|^{p/(p-1)} + C.
\end{align*}
From the preceding inequalities, since the Gaussian measure $\nu$ has finite moments of every order, it is clear that $\Lambda$ satisfies the integrability condition \eqref{strongexpon}.
\end{proof}

We define the following analog of the log mgf in the case of weighted sums,
\begin{equation}\label{psidef}
\Psi(t) \doteq \int_\R \Lambda(tu) \tfrac{1}{\sqrt{2\pi}} e^{-u^2/2}du, \quad t\in \R.
\end{equation}
Our first main result is the following.

\begin{theorem}[Weighted LDP]\label{th-qldp}
Assume \eqref{sigproj} and \eqref{strongexpon}. Then, for  $\sigma$-a.e.\  $\theta \in \seq$,   the sequence $(W_\theta^{(n)})_{n\in\N}$ of \eqref{weightproj} satisfies an LDP with the convex good rate function $\ir_\sigma$, given by
\begin{equation}\label{irsidef}
  \ir_\sigma(w) \doteq \Psi^*(w) = \sup_{t\in\R} \{ tw - \Psi(t)\}.
\end{equation}
\end{theorem}

The proof of Theorem \ref{th-qldp} is given in Section \ref{sec-sur}, with intermediate steps established in Section \ref{ssec-surf} and Section \ref{ssec-expequiv}, and the proof completed in Section \ref{ssec-pfqldp}.

In principle, the rate function $\ir_\sigma$ of Theorem \ref{th-qldp} could depend on the particular choice of $\theta$, but our result shows that the rate function is the same for $\sigma$-a.e.\ $\theta$. In the case where $\sigma$ is the product measure $\sigma=\bigotimes_{n\in \N} \sigma_{n-1}$, this follows immediately from the Kolmogorov zero-one law. That is, let $\mathcal{T}_n$ be the sigma-algebra generated by $(\theta^{(k)})_{k\ge n}$, and let 
\begin{equation}\label{taildef}
\mathcal{T} \doteq \bigcap_{n=1}^\infty \mathcal{T}_n  
\end{equation}
denote the \emph{tail sigma-algebra} induced by $(\theta^{(1)},\theta^{(2)},\dots)$. The rate function $\ir_\sigma$ is measurable with respect to $\mathcal{T}$, and the Kolmogorov zero-one law states that $\mathcal{T}$ is trivial under the product measure. Hence, $\ir_\sigma$ coincides for $\sigma$-a.e.\ $\theta\in\seq$. However, our claim holds for general $\sigma$ satisfying \eqref{sigproj}. In particular, Example \ref{ex-sigma}(ii) gives an example of $\sigma$ such that $\theta^{(1)},\theta^{(2)},\dots$ are highly dependent, $\mathcal{T}$ is not trivial, and hence, the lack of dependence of the rate function $\ir_\sigma$ on $\theta$ is not \emph{a priori} obvious.

Given the $\sigma$-a.e.\ statement of Theorem \ref{th-qldp}, it is natural to ask what happens on the set of measure zero in $\seq$ where the stated LDP does not hold. In particular, our second main result Theorem \ref{th-atyp} shows that under certain additional conditions on $\Lambda$, the sequence of directions $\iota$ associated with Cram\'er's theorem is exceptional, in the sense that Cram\'er's rate function $\ir_\iota$ differs from the universal rate function $\ir_\sigma$. For the following theorem, we assume $\gamma$ is symmetric, or specifically:
\begin{equation}
\forall\, t\in \R, \quad \Lambda(t)=\Lambda(-t). \tag{H3} \label{symm}
\end{equation}

\begin{theorem}[Atypicality]\label{th-atyp}
Assume $\Lambda$ satisfies  \eqref{symm}, and let $\ir_\iota$ and $\ir_\sigma$ be given by \eqref{iriodef} and \eqref{irsidef}, respectively. 
\begin{enumerate}[label=\emph{\alph*.}]
\item If $\Lambda \circ \sqrt{\cdot}$ is concave on $\R_+$, then $\ir_\sigma(w) \ge \ir_\iota(w)$ for all $w\in \R$.
\item If $\Lambda \circ \sqrt{\cdot}$ is convex on $\R_+$, then  $\ir_\sigma(w) \le \ir_\iota(w)$  for all $w\in \R$. 
\item If $\Lambda \circ \sqrt{\cdot}$ is concave or convex, but not linear, on $\R_+$, then $\ir_\sigma(w) = \ir_\iota(w) < \infty$ if and only if $w=0$. 
\end{enumerate}
\end{theorem}

The proof of Theorem \ref{th-atyp} is given in Section \ref{sec-atyp}.

\medskip

We now provide some sufficient conditions (established in \cite{barthe2003extremal}) for the convexity or concavity conditions of Theorem \ref{th-atyp} to hold.

\begin{proposition} \label{prop-con}
Assume  the exponential moment condition \eqref{weakexpon} and the symmetry condition \eqref{symm}.
\begin{enumerate}[label=\emph{\roman*.}]
\item  Suppose $\gamma\ne\delta_0$, the Dirac mass at 0. Define $\varphi:\N\ra\R$ by
\begin{equation*}
  \varphi(k) \doteq (2k+1) \frac{\E[|X_1|^{2k}]}{\E[|X_1|^{2k+2}]}, \quad k\in \N.
\end{equation*}
If $\varphi$ is non-decreasing (resp., non-increasing), then $\Lambda \circ \sqrt{\cdot}$ is concave (resp., convex) on $\R_+$.

\item Suppose $\gamma$ has density $f$ such that $\log f \circ \sqrt{\cdot}$ is concave (resp.,  convex) on $\R_+$. Then $\Lambda \circ \sqrt{\cdot}$ is concave (resp., convex) on $\R_+$.
\end{enumerate}
\end{proposition}
\begin{proof}
Part i.\ is established in Theorem 7 of \cite{barthe2003extremal}. Part ii.\ follows from applying Theorem 12 of \cite{barthe2003extremal} with their $f$ replaced by our $f\circ \sqrt{\cdot}$, and noticing that the integrability of $f\circ\sqrt{\cdot}$ follows from the fact that $f$ has finite first moment, due to the exponential moment condition of \eqref{weakexpon}.
\end{proof}

\begin{example}
Suppose $\gamma$ is the \emph{generalized normal distribution}
with location 0, scale $\alpha > 0$, and shape $\beta >
1$; that is, $\gamma=\mu_{\alpha,\beta}$, where 
\begin{equation*}
  \mu_{\alpha,\beta}(dx) \doteq \frac{1}{2\alpha\Gamma(1+\frac{1}{\beta})}e^{-(|x|/\alpha)^\beta}dx
\end{equation*}
It follows from Lemma \ref{lem-tail} that  $\mu_{\alpha,\beta}$ satisfies \eqref{strongexpon}, which implies \eqref{weakexpon}. It is also easy to see that $\mu_{\alpha,\beta}$ satisfies \eqref{symm}. Thus, the conditions of Proposition \ref{prop-con} are satisfied. It follows immediately from  Proposition \ref{prop-con}(ii) that  for $\beta \ge 2$ (resp., for $\beta \le 2$), $\Lambda \circ \sqrt{\cdot}$ is concave (resp., convex). In fact, for $\beta \ne 2$, the concavity (resp., convexity) is strict.
\end{example}

The preceding example suggests the particular role of the Gaussian, which corresponds to $\beta =2$.  In particular, $\gamma=\mu_{\alpha,2}$ for some $\alpha > 0$ if and only if $\Lambda\circ \sqrt{\cdot}$ is linear. Thus, we could interpret the conditions of Theorem \ref{th-atyp} as evaluating whether our distribution of interest is ``more" or ``less" log-concave than the Gaussian. We also have the following result in the Gaussian case (i.e., when $\gamma=\mu_{\alpha,2}$), which holds for \emph{all} $\theta$ as opposed to just for $\sigma$-a.e.\ $\theta$.

\begin{proposition}\label{prop-gsn}
Suppose $\gamma = \mu_{\alpha,2}$ for some $\alpha > 0$. Then, for all $\theta \in \seq$, the sequence $(W_\theta^{(n)})_{n\in \N}$ satisfies an LDP with the good rate function $\Psi^*(w) = \Lambda^*(w) = (w/\alpha)^2$, where $\Lambda^*$ is defined in \eqref{legendef} with $\Lambda$ the log mgf of the Gaussian with mean 0 and variance $\alpha^2/2$.
\end{proposition}

\begin{proof}
This follows from the fact that for all $n\in \N$, the Gaussian measure on $\R^n$ is spherically symmetric, and hence, for any $\theta^{(n)} \in \sphn$, the law of $\langle X^{(n)}, \theta^{(n)}\rangle_n$ is the same as the law of $\langle X^{(n)}, \onen\rangle_n$. Thus, the LDP for $(W_\theta^{(n)})_{n\in\N}$ follows from the classical Cram\'er's theorem for empirical means of i.i.d.\ Gaussians, for which the rate function can be easily computed to be $\Lambda^*(w) = (w/\alpha)^2$.
\end{proof}

\begin{remark}
It is not clear whether a converse of Proposition \ref{prop-gsn} holds. That is, whether $\ir_\sigma \equiv \ir_\iota$ if and only if $\gamma$ is Gaussian. As one possible approach in this direction, it would be sufficient to show that for any measure $\gamma$ satisfying both \eqref{strongexpon} and \eqref{symm} (and possibly some additional natural conditions), the function $\Lambda \circ \sqrt{\cdot}$ must be either concave or convex.
\end{remark}

Aside from the sequence of Cram\'er directions $\iota\in\seq$, another  natural sequence of directions to consider is the sequence of canonical basis vectors, $e_1 = (e_1^{(1)}, e_1^{(2)},\dots) \in \seq$, where
\begin{equation*}
  e_1^{(n)}  \doteq (1,\overbracket{0,\dots, 0}^{n-1 \text{ times }}) \in \sphn.
\end{equation*}
Then $W_{e_1}^{(n)} = X_1/\sqrt{n}$ for all $n$. The following result states that under certain tail conditions, such normalized projections yield a trivial LDP, again with a rate function different from $\ir_\sigma$.

\begin{proposition}\label{prop-e1}
Assume the following condition (which is stronger than \eqref{strongexpon}):
\begin{equation} \tag{H2$'$} \label{strongalt}
\exists \, C < \infty, \, r\in[0,2)  \textnormal{ such that } \forall\,t\in\R, \quad  \Lambda(t) \le C(1+|t|^r).
\end{equation}
Then the sequence $(W_{e_1}^{(n)})_{n\in \N}$ satisfies an LDP with the trivial good rate function $\chi_{0}$ given by
\begin{equation*}
  \chi_0(x) \doteq \left\{\begin{array}{ll}
 0, & x=0;\\
 \infty, & x\ne 0.
 \end{array}\right.
\end{equation*}
\end{proposition}

\begin{proof}
Consider the limit log mgf associated with the G\"artner-Ellis theorem (recalled for convenience later in Theorem \ref{th-ge}). For all $t\in \R$,
\begin{equation*}
\Lambda_n(t) \doteq \frac{1}{n} \log \E[\exp(tnW_{e_1}^{(n)})] = \frac{1}{n}\log\E[\exp(t\sqrt{n} X_1)] = \frac{1}{n} \Lambda(t\sqrt{n}) \le \frac{1}{n} (C|t|^{r} n^{r/2} + C).
\end{equation*}
Since the exponent $r$ of \eqref{strongalt} satisfies $r<2$ by assumption, we have $\lim_{n\ra\infty}\Lambda_n(t) = 0$ for all $t\in \R$. Thus, by the G\"artner-Ellis theorem, the sequence $(W_{e_1}^{(n)})_{n\in\N}$ satisfies an LDP with good rate function $0^*=\chi_0$.
\end{proof}

\subsection{Relation to prior work}

There is a wealth of literature on large deviations for weighted sums, but our work seems to be the first to emphasize the unique position of Cram\'er's theorem in the geometric setting. Moreover, it appears that none of the existing literature is readily adaptable to our particular problem. We offer a partial (but inevitably, incomplete) survey of existing results. 

In the  somewhat classical works of Book, \cite{book1972large} and \cite{book1973large}, we can find asymptotics bounds for quantities of the form
\begin{equation*}
  \P\left( \frac{\sum_{k=1}^n a_{nk} X_k}{\sum_{k=1}^n a_{nk}} > c  \right),
\end{equation*}
where $(a_{nk})_{k\le n, n\in \N}$ is a triangular array of weights such that $\sum_{k=1}^n a_{nk}^2 = 1$ for all $n$.  However, this does not address our setting because if we let $a_{nk} = \thn_k$, we have  $\sum_{k=1}^n a_{nk}^2 = 1$, but this only yields tail bounds of the form $  \P( W_\theta^{(n)} > cn^{-1/2}\sum_{k=1}^n \theta_k^{(n)})$, as opposed to the desired asymptotics for $\P( W_\theta^{(n)} > c)$. Furthermore, Book does not establish an LDP or identify a rate function.

In a more recent line of work, consider \cite{kiesel2000large}, where on their p.\ 932,  their $\lambda$ and $\nu$ correspond to our $n$ and $k$, respectively. For $Z\sim N(0,1)$, we have the correspondence:
\begin{align*}
a_j(n) &= \mathbf{1}_{\{j\le n-1\}} \frac{1}{n} \sqrt{n} \theta_j^{(n)}, \quad j, n\in \N;\\
\sum_{j=0}^\infty a_j(n)^k &\approx \sum_{j=0}^{n-1} \frac{1}{n^k} z_j^k \stackrel{(\zeta\textnormal{-a.e.})}{\approx} \frac{\E[|Z|^k]}{n^{k-1}}  \doteq \frac{a_k}{n^{k-1}} , \quad k,n\in \N;\\
\phi(n) &= n, \quad n\in \N.
\end{align*}
Suppose that the sequence $(a_k)_{k\in \N}$ (which depends on the particular choice of weights $a_j(n)$, $j,n\in \N$) satisfies the following condition (from p.\ 932 of \cite{kiesel2000large}):
\begin{equation}\label{acond}
   \lim_{k\ra \infty} |a_k|^{1/k} < \infty.
\end{equation}
The main result of \cite{kiesel2000large} is that for a sequence of i.i.d.\ random variables $(X_k)_{k\in \N}$ with cumulants $(c_k)_{k\in\N}$, if condition \eqref{acond} holds, then the sequence of weighted means $\frac{1}{n}\sum_{j=1}^n a_j(n) X_j$, $n\in\N$, satisfies an LDP with rate function $\chi^*$, the Legendre transform of $\chi(t) \doteq \sum_{k=2}^\infty \frac{a_kc_k}{k!}t^k$. However, the finiteness condition \eqref{acond} does not hold in our setting of $a_k = \E[|Z|^k]$, since the following limit is infinite:
\begin{equation*}
\lim_{k\ra \infty} \E[|Z|^k]^{1/k} = \sqrt{2} \lim_{k\ra \infty} \Gamma(\tfrac{k+1}{2})^{1/k} = \infty.
\end{equation*}
Therefore, the weighted mean LDP of \cite{kiesel2000large} does not apply in our setting.

Yet more recently, \cite{najim2002cramer} proves an LDP for weighted empirical means similar to \eqref{weightproj}, except with weights that are uniformly bounded (in $n$). Our results correspond to \emph{unbounded} weights $\sqrt{n}\thn_i$ which are not covered by their results. Similarly, \cite{chi2001stochastic} proves an LDP for empirical means of certain bounded functionals, which again fails to apply to our unbounded weights.

In the context of information theory, \cite{dembo2002source} states an LDP for sums of the form $\frac{1}{n}\sum_{i=1}^n \rho(x_i,Y_i)$, where $(x_i)_{i\in\N}$ are ``weights", $(Y_i)_{i\in \N}$ is a sequence of random variables satisfying certain mixing properties, and $\rho:\mathcal{X}\times\mathcal{Y}\ra\R_+$ for Polish spaces $\mathcal{X}$ and $\mathcal{Y}$. The LDP is stated in the form of a \emph{generalized asymptotic equipartition property} for ``distortion measures". However, note that $\rho$ is assumed to be nonnegative, so a function like $\rho(x,y)=xy$ (corresponding to projections) does not fit within the setting of \cite{dembo2002source}. Moreover, their weights $(x_i)_{i\in \N}$ are assumed to be a realization of a stationary ergodic process, which is not the case for our weights $\sqrt{n}\thn$ that are drawn from the scaled sphere $\sqrt{n}\sphn$. This lends our work a geometric rather than information-theoretic interpretation.

The paper \cite{gantert2014large}, co-authored by the first and third authors of this work, also analyzes weighted sums of i.i.d.\ random variables, but there the emphasis is on sums of subexponential random variables, rather than the weights themselves.

The most closely related work to our own is the recent work of \cite{bovier2014conditional}, which gives strong large deviations (i.e., refined asymptotics) for weighted sums of i.i.d.\ random variables and i.i.d.\ weights, conditioned on the weights. Our weights $\sqrt{n}\thn_i$ are not i.i.d., but in Section \ref{ssec-surf} and Section \ref{ssec-expequiv}, we prove that Theorem \ref{th-qldp} can be reduced to an LDP for the sequence $(\widehat{W}_z^{(n)})_{n\in\N}$ defined in \eqref{whatdef}, which is an i.i.d.\ weighted sum, conditional on given weights. With some additional calculations from this point, the rate function $\ir_\sigma$ of Theorem \ref{th-qldp} could then be deduced from the conditional LDP of \cite{bovier2014conditional}, stated in their Theorem 1.6  with rate function defined in their equation (1.13). Note that condition (iii) of their Theorem 1.6 has two parts, but our integrability condition \eqref{strongexpon} corresponds only to their first part; in fact, it follows from Lemma \ref{lem-psidiff} that their second part follows from our condition \eqref{firstexpon}, which is weaker than \eqref{strongexpon}, and thus, need not be assumed separately. Moreover, our research (completed independently) differs due to our emphasis on a geometric point of view; as a consequence, we can explicitly identify a rate function $\ir_\sigma$ and highlight the atypical position occupied by Cram\'er's theorem.

Lastly, the methods we use are a simplification of those developed in a companion paper \cite{gkr2}, where we consider normalized projections of certain \emph{non}-product measures, as well as projections in random directions.

\section{The $\sigma$-almost everywhere\ LDP}\label{sec-sur}

\subsection{The surface measure on $\sphn$}\label{ssec-surf}\label{ssec-surfmeas}

In this section, we recall a convenient representation for a random vector distributed according to the surface measure on $\sphn$, in order to obtain \eqref{sigzet1}, which reduces $\sigma$-a.e.\ statements into more tractable statements about Gaussian random variables. Let $\mathbb{A} \doteq \prod_{n\in \N} \R^n$ denote the space of infinite triangular arrays. That is, $z\in\mathbb{A}$ is of the form $z=(z^{(1)},z^{(2)},\dots)$ where $z^{(n)} \in \R^n$ for all $n\in \N$. Let $\mathcal{R}:\mathbb{A}\rightarrow\mathbb{A}$ be the map such that for $z \in \mathbb{A}$, the $n$-th row of $\mathcal{R}(z)$ is
\begin{equation*}
  [\mathcal{R}(z)]^{(n)} \doteq \frac{z^{(n)}}{\|z^{(n)}\|_{n}}.
\end{equation*}
Let $\bar\pi_n:\mathbb{A}\rightarrow \mathbb{R}^n$ denote the $n$-th row map such that $\bar\pi_n(z) = z^{(n)}$. Let $\nu$ denote the Gaussian measure on $\R$, and let $\nu^{\otimes n}$ denote the standard Gaussian measure on $\R^n$.

\begin{lemma}\label{lem-equivsigzet}
If $\zeta\in\mathcal{P}(\mathbb{A})$ is such that 
\begin{equation} \label{zetdef}
  \zeta \circ \bar\pi_n^{-1} = \nu^{\otimes n}, \quad n\in \N,
\end{equation}
then $\sigma\doteq \zeta \circ \mathcal{R}^{-1}$ satisfies \eqref{sigproj}. Conversely, if $\sigma\in \mathcal{P}(\seq)$ satisfies \eqref{sigproj}, then there exists some $\zeta\in \mathcal{P}(\mathbb{A})$ satisfying \eqref{zetdef} such that $\sigma = \zeta \circ \mathcal{R}^{-1}$.
\end{lemma}
\begin{proof}
Both results are merely a restatement of the well known fact that if $Z^{(n)}$ has the $n$-dimensional standard  Gaussian distribution, then $Z^{(n)}/\|Z^{(n)}\|_n$ is uniformly distributed on the unit sphere $\sphn$, and independent of $\|Z^{(n)}\|_n$. 

%
\end{proof}

Note that Lemma \ref{lem-equivsigzet} states that for any given $\sigma \in \mathcal{P}(\seq)$, we can find a corresponding $\zeta\in \mathcal{P}(\mathbb{A})$. Fix such a pair $(\sigma,\zeta)$. Now, for $z\in \mathbb{A}$, define
\begin{equation}\label{whatdef}
\widehat{W}_z^{(n)} \doteq \frac{1}{n} \sum_{i=1}^n X_i\,z_i^{(n)}.
\end{equation}
Then, given $W_\theta^{(n)}$ as defined in \eqref{weightproj}, and any good rate function $\ir:\R\ra[0,\infty]$,  Lemma \ref{lem-equivsigzet} implies that 
\begin{align}
  \sigma &\left(\theta\in \seq : (W_\theta^{(n)})_{n\in \N} \text{ satisfies an LDP  with good rate  function } \ir \right) \notag\\
    &\quad = \zeta\left( z\in \mathbb{A} : (\tfrac{\sqrt{n}}{\|z^{(n)}\|_n} \widehat{W}_z^{(n)})_{n\in \N} \text{ satisfies an LDP  with good rate function } \ir \right). \label{sigzet1}
\end{align}
In addition, Lemma \ref{lem-equivsigzet} yields a large class of examples of $\sigma$ satisfying \eqref{sigproj}, constructed via $\zeta$ satisfying \eqref{zetdef}. We specify two such examples below.

\begin{example}\label{ex-sigma} \quad
\begin{enumerate}[label=\alph*.]
\item Consider the completely independent case, where the elements
  $Z_i^{(n)}$, $i=1,\dots, n$, $n\in \N$, are all independent; then
  the law of $\mathcal{R}(Z)$ is the product measure $\sigma =
  \bigotimes_{n\in \N} \sigma_{n-1}$, where each row $\theta^{(n)}$ 
of $\theta$ is independent under $\sigma$.   As previously noted, the tail sigma-algebra $\mathcal{T}$ induced by the rows (defined in \eqref{taildef}), is trivial in this case due to the Kolmogorov zero-one law.
\item Alternatively, consider the following highly dependent case: let
  $\zeta\in\mathcal{P}(\mathbb{A})$ satisfy \eqref{zetdef} such that for $\zeta$-a.e.\ $z\in \mathbb{A}$, we have
  $z_i^{(n)} = z_i^{(m)}$ for all $i\in \N$ and $m,n\ge i$ (i.e.,
  constant within columns). Then, let $\sigma = \zeta \circ
  \mathcal{R}^{-1}$, so that $\sigma$ satisfies \eqref{sigproj} by Lemma \ref{lem-equivsigzet}. In this case, there is strong dependence across rows which precludes a claim regarding triviality of the tail sigma-algebra $\mathcal{T}$ induced by the rows. In fact, consider the event 
\begin{equation*}
A \doteq \left\{\theta\in\seq:  \lim_{n\ra\infty} \sqrt{n} \thn_1 > 0 \right\}
\end{equation*}
Note that $A$ is measurable with respect to 
$\mathcal{T}$.
 However, we also have due to the strong law of large numbers
 ($\zeta$-a.e., as stated precisely in \eqref{zslln}),
\begin{align*}
\sigma(A) = \zeta\left( z\in \mathbb{A} : \lim_{n\ra\infty} \sqrt{n} \,z_1^{(n)} / \|z^{(n)}\|_{n,2} > 0\right) = \zeta\left( z\in \mathbb{A} : z_1^{(1)}  > 0 \right) = \frac{1}{2}.
\end{align*}
That is, $\mathcal{T}$ is non-trivial, and so $\ir_\sigma$ 
cannot {\em a priori} be declared as $\sigma$-a.e.\ constant 
through a simple analysis of the tail sigma-algebra.  
\end{enumerate}

\end{example}

\begin{remark}\label{rmk-sigma}
We assume the condition \eqref{sigproj} not in an attempt to be as
general as possible, but rather to point out that the universality of
the rate function is a genuinely interesting phenomenon. Specifically,
if we only consider the independent case of Example \ref{ex-sigma}a.,
then the fact that $\ir_\sigma$ is ``universal" (in that it does not
depend on $\theta$) is a consequence of the fact that  the tail
sigma-algebra $\mathcal{T}$ is trivial.
However, Example \ref{ex-sigma}b.\ shows that  universality of the rate function is a more general phenomenon that holds even when $\mathcal{T}$ is non-trivial. The condition \eqref{sigproj} only imposes constraints on the ``marginal" distribution of the $n$-th row of  the array $\theta$, and imposes no restrictions on the dependence across different rows $\thn$, $n\in\N$. In fact, for  $Z\sim \zeta$ satisfying \eqref{zetdef}, the elements of $Z$ need not even be \emph{jointly} Gaussian in order for the law of $\mathcal{R}(Z)$ to satisfy \eqref{sigproj}.
\end{remark}

\subsection{Exponential equivalence}\label{ssec-expequiv}

As a consequence of Lemma \ref{lem-equivsigzet} and the equality in \eqref{sigzet1}, we can replace $\sigma$-a.e.\ statements about $W_\theta^{(n)}$, $n\in\N$ with $\zeta$-a.e.\ statements about $(\sqrt{n}/ \|z^{(n)}\|_n)\widehat{W}_z^{(n)}$, $n\in\N$. In this section, we go further and explain why in the large deviations setting, we can ignore the contribution of the multiplicative factor $\sqrt{n}/\|z^{(n)}\|_n$. That is, we show that such a factor yields an exponentially equivalent sequence, defined as follows.

\begin{definition}\label{def-expequiv}
Let $(\xi_n)_{n\in\N}$  and $(\tilde{\xi}_n)_{n\in\N}$ be two sequences of $\R$-valued random variables such that for all $\delta > 0$,
\begin{equation*}
  \limsup_{n\ra\infty} \frac{1}{n} \log \P( |\xi_n - \tilde{\xi}_n| > \delta ) = -\infty;
\end{equation*}
then $(\xi_n)_{n\in\N}$ and $(\tilde{\xi}_n)_{n\in\N}$ are said to be \emph{exponentially equivalent}.
\end{definition}

\begin{proposition}[\cite{DemZeibook}]
\label{prop-expeq}
If $(\xi_n)_{n\in\N}$ is a sequence of random variables that satisfies an LDP with   good rate function $\ir$, and $(\tilde{\xi}_n)_{n\in\N}$ is another sequence that is exponentially equivalent to $(\xi_n)_{n\in\N}$, then $(\tilde{\xi}_n)_{n\in\N}$ satisfies an LDP with good rate function $\ir$.
\end{proposition}

\begin{lemma}\label{lem-expon} Let $(\xi_n)_{n\in\N}$ be a sequence of random variables that satisfies an LDP with a good rate function $\ir$. Let $(a_n)_{n\in\N}$ be a deterministic sequence such that $a_n\ra 1$ as $n\rightarrow\infty$, and let $(\tilde{\xi}_n)_{n\in\N}$ be another sequence defined by:
\begin{equation*}
  \tilde{\xi}_n = a_n {\xi}_n, \quad n\in\mathbb{N}.
\end{equation*}
If $\ir$ is quasiconvex --- that is, if the set $\{x \in \R : \ir(x) \in (-\infty,c) \}$ is convex for all $c\in \R$ --- then $({\xi}_n)_{n\in\N}$ and $(\tilde{\xi}_n)_{n\in\N}$ are exponentially equivalent. 

\end{lemma}

\begin{proof}
For $\epsilon > 0$, let $N_\epsilon<\infty$ be such that for all $n\ge N_\epsilon$, we have $|1-a_n| < \epsilon$. For $n\ge N_\epsilon$ and any $\delta > 0$,
\begin{equation*}
|\tilde{\xi}_n - {\xi}_n| \ge \delta \quad \Leftrightarrow \quad |\xi_n| \cdot |1- a_n| \ge \delta \quad \Rightarrow \quad |\xi_n| \ge \tfrac{\delta}{\epsilon}.
\end{equation*}
Because $\ir$ is lower semicontinuous and has compact level sets, it achieves its global minimum at some (not necessarily unique) $\bar x \in \R$. Fix $\delta > 0$ and let $\epsilon > 0$ be small enough such that $|\bar x| < \tfrac{\delta}{\epsilon}$. Then, 
\begin{align*}
  \limsup_{n\ra\infty} \frac{1}{n} \log \P( |\tilde{\xi}_n - \xi_n| > \delta ) &\le  \limsup_{n\ra\infty} \frac{1}{n} \log \P( | \xi_n| \ge \tfrac{\delta}{\epsilon} ) \\
  &\le -\inf_{|x|\ge \delta/\epsilon} \ir(x)\\
  &= -\min\left[\ir(\tfrac{\delta}{\epsilon}) ,  \, \ir(-\tfrac{\delta}{\epsilon})\right].
\end{align*}
The second inequality follows from the LDP for $(\xi_n)_{n\in \N}$. The last equality follows from the fact that if a quasiconvex function has a global minimizer $\bar x$, then it is non-increasing for $x < \bar x$, and non-decreasing for $x > \bar x$ \cite[Lemma 1]{luenberger1968quasi}. Hence, since the rate function $\ir$ is quasiconvex and has a global minimizer $\bar x$ which satisfies $|\bar x| < \delta/\epsilon$, it follows  that if $x \ge \delta/\epsilon$ (resp., $x \le -\delta/\epsilon$), then we have $\ir(x) \ge \ir(\delta/\epsilon)$ (resp., $\ir(x) \ge \ir(-\delta/\epsilon)$). Lastly, take the limit as $\epsilon\ra 0$, and use the compactness of the level sets of  $\ir$ to conclude that $\ir(\frac{\delta}{\epsilon}) \ra +\infty$ and $\ir(-\frac{\delta}{\epsilon}) \ra +\infty$. This proves the required exponential equivalence.
\end{proof}

Fix $\zeta$ satisfying \eqref{zetdef}. Due to the strong law of large numbers, we have that for $\zeta$-a.e.\ $z\in \mathbb{A}$,
\begin{equation}\label{zslln}
  \frac{\sqrt{n}}{\|z^{(n)}\|_n} = \left(\frac{1}{n} \sum_{i=1}^n (z_i^{(n)})^2 \right)^{-1/2} \xrightarrow{n\ra\infty} 1.
\end{equation}
Thus, we are in a prime position to apply Lemma \ref{lem-expon}, which motivates the analysis of an LDP for $(\widehat W_z^{(n)})_{n\in\N}$.

\subsection{Proof of the LDP for $(W_\theta^{(n)})_{n\in\N}$} \label{ssec-pfqldp}

We aim to prove an LDP for the sequence $(\widehat{W}_z^{(n)})_{n\in \N}$; that is, an LDP for sums of independent but not identically distributed random variables (where the lack of identical distribution comes from the inhomogeneous weights $z_i^{(n)}$ within the sum). The G\"artner-Ellis theorem (recalled below) is well suited for such an LDP.

\begin{theorem}[G\"artner-Ellis]\label{th-ge}
Let $(\xi_n)_{n\in\N}$ be a sequence of $\R$-valued random variables. Suppose that the  \emph{limit log mgf} $\bar{\Lambda}:\R\ra[0,\infty)$ defined by
\begin{equation*}
  \bar\Lambda(t) \doteq \lim_{n\ra\infty} \frac{1}{n} \log \E[e^{tn\xi_n}]
\end{equation*}
is finite and differentiable at all $t\in \R$. Then $(\xi_n)_{n\in \N}$ satisfies an LDP with the  convex good rate function $\bar\Lambda^*$, the Legendre transform of $\bar\Lambda$.
\end{theorem}

For a proof of Theorem \ref{th-ge}, we refer to \cite[Theorem V.6]{den2008large}, which also includes a more general version of the G\"artner-Ellis theorem that applies even if $\bar\Lambda$ is finite for only \emph{some} $t\in \R$ (under mild additional conditions).

The following lemma establishes a property of $\Psi$ which will be used in the application of the G\"artner-Ellis theorem.

\begin{lemma}\label{lem-psidiff}
Suppose that
\begin{equation}\label{firstexpon}
\forall \, t\in \R, \quad \int_\R |\Lambda(tu)| \, \nu(du) < \infty.
\end{equation}
Then, the function $\Psi$ of \eqref{psidef} is differentiable on $\R$.
\end{lemma}
\begin{proof}
For each $t\in \R$, differentiability of $\Psi$ at $t$ follows from the differentiability of $t\mapsto \Lambda(tu)$ for all $u\in \R$, and an application of the dominated convergence theorem with the dominating function
\begin{equation*}
  g_t(u) \doteq |\Lambda'((t-1)u)u| + |\Lambda'((t+1)u)u|, \quad u\in \R.
\end{equation*}
Indeed, fix $t\in \R$ and for each $\delta \in (-1,1)$ and $u\in\R$, define the difference quotient $R_{t,\delta}(u) \doteq [\Lambda((t+\delta)u) - \Lambda(tu)]/\delta$. Then,
\begin{equation*}
  |R_{t,\delta}(u)|  \le \sup \left\{ |\Lambda'((t+\alpha)u)u| : \alpha\in[-1,1]\right\} \le g(u),
\end{equation*}
where the last inequality uses the fact that $t\mapsto u\,\Lambda'(tu)$ is monotone. To show that $g_t$ is integrable, first note that the convexity of $\Lambda$ implies that for $u,s\in \R$,
\begin{equation*}
  \Lambda(su) - \Lambda(0) \le \Lambda'(su)\,su \le \Lambda(2su)-\Lambda(su),
\end{equation*}
and hence,
\begin{equation*}
  |\Lambda'(su)su| \le |\Lambda(0)| + |\Lambda(su)| + |\Lambda(2su)|.
\end{equation*}
Since, by the assumption \eqref{firstexpon}, for every $s\in \R$, the right-hand side is an integrable function of $u$, it follows that $g_t$ is also integrable for every $t\in \R$.

\end{proof}

\begin{proof}[Proof of Theorem \ref{th-qldp}]
Due to Lemma \ref{lem-equivsigzet} (in particular, its consequence, \eqref{sigzet1}), it suffices to prove a $\zeta$-a.e.\ LDP for the sequence $((\sqrt{n}/\|z\|_n)\widehat{W}_z^{(n)})_{n\in\N}$, where $\widehat{W}_z^{(n)}$ is defined as in \eqref{whatdef}. Due to Lemma \ref{lem-expon} and the limit \eqref{zslln}, it suffices to prove a $\zeta$-a.e.\ LDP for $(\widehat{W}_z^{(n)})_{n\in\N}$. To this end, we consider the G\"artner-Ellis limit log mgf for the sequence $(\widehat{W}_z^{(n)})_{n\in\N}$. For every $n\in \N$ and $t\in \R$, we have due to the independence of $X_i$, $i=1,\cdots, n$,
\begin{equation}\label{gepres}
\Lambda_{n,z}(t)\doteq  \frac{1}{n} \log \E\left[ \exp\left(tn\widehat{W}_z^{(n)}\right) \right] =  \frac{1}{n} \log \prod_{i=1}^n \E\left[ \exp\left(t X_iz_i^{(n)}\right) \right]  =  \frac{1}{n} \sum_{i=1}^n \Lambda(tz_i^{(n)}). 
\end{equation}
We first claim that for $\zeta$-a.e.\ $z\in \mathbb{A}$,  the G\"artner-Ellis limit log mgf, the limit of \eqref{gepres}, satisfies, for each $t\in \R$,
\begin{equation}\label{gelim}
 \lim_{n\ra\infty} \Lambda_{n,z}(t) = \int_\R \Lambda(tu) \nu(du) = \Psi(t),
\end{equation}
with $\Psi$ as defined in \eqref{psidef}.

We proceed by proving the following modified claim (obtained by interchanging the quantifiers in our original claim):  for each $t\in \R$, for $\zeta$-a.e. $z\in \mathbb{A}$, the expression \eqref{gelim} holds. Note that if $z$ were an i.i.d.\ sequence instead of a triangular array, our modified claim would follow from the usual strong law of large numbers. However, the strong LLN does not necessarily extend to empirical means of rows of i.i.d.\ random variables in a triangular array (see, e.g., \cite[Example 5.41]{romano1986counterexamples}). On the other hand, if the common distribution of the i.i.d.\ elements (in our case,  each of the random variables $\Lambda(tz_i^{(n)})$, $i=1,\dots,n$, $n\in \N$) has finite fourth moment, then the strong LLN follows from a standard weak LLN and Borel-Cantelli argument \cite[p.113, (i)]{romano1986counterexamples}.  Due to our assumption \eqref{strongexpon}, it follows that for all $t\in \R$, for $\zeta$-a.e. $z\in \mathbb{A}$, the limit \eqref{gelim} holds.

Next, we aim to interchange the quantifiers to establish the original claim. Note that for each $n\in\N$,  $\Lambda_{n,z}$ of \eqref{gepres} is a convex function (since it is the sum of convex functions).  Now, let $T\subset \R$ be countable and dense. Then, it follows from countable additivity that for $\zeta$-a.e. $z\in \mathbb{A}$, the convex functions $\Lambda_{n,z}(t)$ converge pointwise as $n\ra\infty$ to $\Psi(t)$, for all $t$ in the dense subset $T\subset\R$. Hence,  the convex analytic considerations of \cite[Theorem 10.8]{rockafellar1970convex} imply that the pointwise convergence of $\Lambda_{n,z}(t)$ to $\Psi(t)$ holds for all $t\in\R$. That is, for $\zeta$-a.e.\ $z\in \mathbb{A}$, for all $t\in \R$, the limit \eqref{gelim} holds, proving our original claim.

Since \eqref{strongexpon} holds, $\Psi(t)<\infty$ for all $t\in \R$ and, because \eqref{firstexpon} follows trivially from \eqref{strongexpon}, Lemma \ref{lem-psidiff} implies that $\Psi$ is differentiable on $\R$. Therefore, by the G\"artner-Ellis Theorem (Theorem \ref{th-ge}), for $\zeta$-a.e.\ $z\in \mathbb{A}$,  the sequence $(\widehat{W}_z^{(n)})_{n\in\N}$ satisfies an LDP with good rate function $\Psi^*$.
\end{proof}

\section{Atypicality}\label{sec-atyp}

In this section, we compare the rate function $\ir_\sigma$ with the Cram\'er rate function $\ir_\iota$. We first use Jensen's inequality to compare the associated log mgfs $\Psi$ and $\Lambda$.

\begin{lemma}\label{lem-mgfineq}
Assume \eqref{symm}, and let $\Lambda$ and $\Psi$ be defined as in \eqref{weakexpon} and \eqref{psidef}, respectively.
\begin{enumerate}[label=\emph{\alph*.}]
\item If $\Lambda \circ \sqrt{\cdot}$ is concave on $\R_+$, then $\Psi(t) \le \Lambda(t)$ for all $t\in \R$.
\item If $\Lambda \circ \sqrt{\cdot}$ is convex on $\R_+$, then  $\Psi(t) \ge \Lambda(t)$  for all $t\in \R$. 
\item If $\Lambda \circ \sqrt{\cdot}$ is concave or convex, but not linear, on $\R_+$, then $\Lambda(t) = \Psi(t)$ if and only if $t=0$.
\end{enumerate}
\end{lemma}

\begin{proof}
We begin with part a. Let $\nu$ be the standard Gaussian distribution, and let $Z\sim \nu$ be a standard Gaussian random variable. Then, for all $t\in \R$, we have
\begin{align*}
\Psi(t) &= \E[\Lambda(tZ)]\\
\text{\small(symmetry)}\quad   &= \E\left[\Lambda\left((t^2Z^2)^{1/2}\right)\right] \\
\text{\small(Jensen)} \quad   &\le \Lambda\left(\E[t^2Z^2]^{1/2}\right) \\
   &= \Lambda(t).
\end{align*}
Similar calculations can be used to establish  part b. As for part c., recall that in Jensen's inequality, equality holds if and only if either: (i) $\Lambda \circ \sqrt{\cdot}$ is linear; or (ii) the underlying random variable is almost surely constant. Note that (i) is not the case by assumption. As for (ii), this holds if and only if $t^2Z^2$ is almost surely constant, which is the case if and only if $t=0$.
\end{proof}

Before we prove the theorem, we recall some basic facts about the log mgf of $X_1\sim \gamma$. Let the \emph{domain} of a function $f:\R\rightarrow\R$ be the set $D_f\doteq \{x\in \R : f(x) < \infty\}$. For a set $D\subset \R$, let $D^\circ$ denote the interior of $D$.

\begin{lemma}\label{lem-lamstan}
Let $\Lambda(t) = \log \E[e^{tX_1}]$ be the log mgf of some random variable $X_1$. Then,
\begin{enumerate}
\item $\Lambda$ is lower semicontinuous;
\item $\Lambda$ is smooth in $D_{\Lambda}^\circ$;
\item $\Lambda$ is convex.
\end{enumerate}
Furthermore, if $X_1$ is non-degenerate (i.e., not a.s.\ constant), then
\begin{enumerate}[resume]
\item $\Lambda$ is strictly convex in $D_{\Lambda}^\circ$;
\item $\Lambda^*$ is differentiable in $D_{\Lambda^*}^\circ$;\item for $x\in D_{\Lambda^*}^\circ$, the maximum in the definition of the Legendre transform is uniquely attained --- that is, the following  quantity is well defined:
\begin{equation}
  t_x \doteq \arg\max\{tx - \Lambda(t)\}. \label{txmax}
\end{equation}
\end{enumerate}
\end{lemma}
\begin{proof}
These are mostly standard, but we provide sketches of the proofs. For 1., lower semicontinuity follows from Fatou's lemma. For 2., smoothness follows from interchanging differentiation and expectation. Convexity in 3.\ and strict convexity in 4.\ follow from H\"older's inequality. As for 5., it is classical that if a function is lower semicontinuous and strictly convex in the interior of its domain, then its Legendre transform is differentiable in the interior of its domain (see \cite[Theorem 26.3]{rockafellar1970convex}). Lastly, for 6., it is also classical that for $x\in D_{\Lambda^*}^\circ$,  we have $t_x = (\Lambda^*)'(x)$ (see \cite[Theorem 26.5]{rockafellar1970convex}).
\end{proof}

\begin{proof}[Proof of Theorem \ref{th-atyp}]
Assume without loss of generality that $X_1$ is non-degenerate. If it were degenerate, then due to the symmetry condition \eqref{symm}, the law of $X_1$ must be that $\gamma = \delta_0$, in which case $\Lambda = \Psi = 0$. Therefore, $\ir_\sigma$ and $\ir_\iota$ are both  equal to the characteristic function at 0 (which is equal to  0 at $w=0$ and $+\infty$ for all other $w$), and the result is trivial.

Suppose $\Lambda \circ \sqrt{\cdot}$ is concave (the convex case is similar, but with inequalities reversed). Due to Lemma \ref{lem-mgfineq}, we have $\Psi(t) \le \Lambda(t)$ for all $t\in \R$, which due to the definition of the Legendre transform implies that $\ir_\sigma(w) = \Psi^*(w) \ge \Lambda^*(w)= \ir_\iota(w)$ for all $w\in \R$, thus proving a.\ (and b.\ for the convex case).

Further assume the stronger condition of c., that $\Lambda \circ \sqrt{\cdot}$ is concave but not linear.  Then, for $w\in\R$ such that $\Lambda^*(w) < \infty$, let $t_w$ be as in \eqref{txmax}, which is well defined due to the non-degeneracy condition of Lemma \ref{lem-lamstan}. Then,
\begin{align*}
\ir_\sigma(w) = \Psi^*(w) &\ge t_ww - \Psi(t_w) \\
   &\ge t_ww - \Lambda(t_w) \\
   &= \Lambda^*(w) = \ir_\iota(w).
\end{align*}
Due to Lemma \ref{lem-mgfineq}, the second inequality  above is an equality if and only if $t_w = 0$, which occurs if and only if $(\Lambda^*)'(w) = 0$. Note that $\Lambda$ is symmetric, so $\Lambda^*$ is also symmetric (by definition of the Legendre transform). Moreover, the smoothness of $\Lambda$  (see Lemma \ref{lem-lamstan}), implies the strict convexity of $\Lambda^*$  within its domain (see \cite[Theorem 26.3]{rockafellar1970convex}). Thus, $(\Lambda^*)'(w) = 0$ if and only if $w=0$. This yields the claim of part c.
\end{proof}

\begin{remark}
In this paper, we address the ``atypical" nature of the directions
$\onen = (1,1,\dots,1)$ associated with Cram\'er's theorem for large
deviations of product measures. But in fact, the notions of
atypicality and universal rate function extend beyond the product
case. In particular, the companion paper \cite{gkr2} establishes LDPs
for random projections of random vectors distributed according to the
uniform measure on $\ell^p$ balls, again with a rate function that
coincides for $\sigma$-a.e.\ sequence of directions, and the 
sequence of directions $\onen = (1,1,\dots,1)$, $n \in \mathbb{N}$,
can be shown to be atypical in that setting as well. 
\end{remark}

\section*{Acknowledgements}
We would like to thank an anonymous referee for helpful feedback on the exposition.


\bigskip
\bibliography{ldpbib}

\begin{thebibliography}{18}

\bibitem{barthe2003extremal}
\begin{barticle}[author]
\bauthor{\bsnm{Barthe},~\bfnm{F}\binits{F.}} \AND
  \bauthor{\bsnm{Koldobsky},~\bfnm{A}\binits{A.}}
(\byear{2003}).
\btitle{Extremal slabs in the cube and the {L}aplace transform}.
\bjournal{Advances in Mathematics}
\bvolume{174}
\bpages{89--114}.
\end{barticle}
\endbibitem

\bibitem{book1972large}
\begin{barticle}[author]
\bauthor{\bsnm{Book},~\bfnm{Stephen~A}\binits{S.~A.}}
(\byear{1972}).
\btitle{Large deviation probabilities for weighted sums}.
\bjournal{The Annals of Mathematical Statistics}
\bvolume{43}
\bpages{1221--1234}.
\end{barticle}
\endbibitem

\bibitem{book1973large}
\begin{barticle}[author]
\bauthor{\bsnm{Book},~\bfnm{Stephen~A}\binits{S.~A.}}
(\byear{1973}).
\btitle{A large deviation theorem for weighted sums}.
\bjournal{Zeitschrift f{\"u}r Wahrscheinlichkeitstheorie und Verwandte Gebiete}
\bvolume{26}
\bpages{43--49}.
\end{barticle}
\endbibitem

\bibitem{bovier2014conditional}
\begin{barticle}[author]
\bauthor{\bsnm{Bovier},~\bfnm{Anton}\binits{A.}} \AND
  \bauthor{\bsnm{Mayer},~\bfnm{Hannah}\binits{H.}}
(\byear{2015}).
\btitle{A conditional strong large deviation result and a functional central
  limit theorem for the rate function}.
\bjournal{ALEA, Latin American Journal of Probability And Mathematical
  Statistics}
\bvolume{12}
\bpages{533-550}.
\end{barticle}
\endbibitem

\bibitem{chernoff52measure}
\begin{barticle}[author]
\bauthor{\bsnm{Chernoff},~\bfnm{H}\binits{H.}}
(\byear{1956}).
\btitle{A measure of asymptotic efficiency for tests of a hypothesis based on
  the sum of observations}.
\bjournal{The Annals of Mathematical Statistics}
\bvolume{27}
\bpages{1-22}.
\end{barticle}
\endbibitem

\bibitem{chi2001stochastic}
\begin{barticle}[author]
\bauthor{\bsnm{Chi},~\bfnm{Zhiyi}\binits{Z.}}
(\byear{2001}).
\btitle{Stochastic sub-additivity approach to the conditional large deviation
  principle}.
\bjournal{The Annals of Probability}
\bvolume{29}
\bpages{1303--1328}.
\end{barticle}
\endbibitem

\bibitem{cramer38sur}
\begin{barticle}[author]
\bauthor{\bsnm{Cram\'{e}r},~\bfnm{H}\binits{H.}}
(\byear{1938}).
\btitle{Sur un nouveau th\'{e}or\'{e}me--limite de la th\'{e}orie des
  probabilit\'{e}s}.
\bjournal{Actualit\'{e}s Scientifiques et Industrielles}
\bvolume{736}
\bpages{5-23}.
\end{barticle}
\endbibitem

\bibitem{dembo2002source}
\begin{barticle}[author]
\bauthor{\bsnm{Dembo},~\bfnm{Amir}\binits{A.}} \AND
  \bauthor{\bsnm{Kontoyiannis},~\bfnm{Ioannis}\binits{I.}}
(\byear{2002}).
\btitle{Source coding, large deviations, and approximate pattern matching}.
\bjournal{IEEE Transactions on Information Theory}
\bvolume{48}
\bpages{1590--1615}.
\end{barticle}
\endbibitem

\bibitem{DemZeibook}
\begin{bbook}[author]
\bauthor{\bsnm{Dembo},~\bfnm{Amir}\binits{A.}} \AND
  \bauthor{\bsnm{Zeitouni},~\bfnm{Ofer}\binits{O.}}
(\byear{1998}).
\btitle{Large Deviations Techniques and Applications},
\bedition{2} ed.
\bpublisher{Springer}.
\end{bbook}
\endbibitem

\bibitem{den2008large}
\begin{bbook}[author]
\bauthor{\bsnm{Den~Hollander},~\bfnm{Frank}\binits{F.}}
(\byear{2008}).
\btitle{Large Deviations}.
\bseries{Fields Institute Monographs}
\bvolume{14}.
\bpublisher{American Mathematical Society}.
\end{bbook}
\endbibitem

\bibitem{feller1970probability2}
\begin{bbook}[author]
\bauthor{\bsnm{Feller},~\bfnm{William}\binits{W.}}
(\byear{1970}).
\btitle{An Introduction to Probability Theory and Its Applications, volume II}.
\bpublisher{John Wiley \& Sons}.
\end{bbook}
\endbibitem

\bibitem{gkr2}
\begin{barticle}[author]
\bauthor{\bsnm{Gantert},~\bfnm{Nina}\binits{N.}},
  \bauthor{\bsnm{Kim},~\bfnm{Steven~Soojin}\binits{S.~S.}} \AND
  \bauthor{\bsnm{Ramanan},~\bfnm{Kavita}\binits{K.}}
(\byear{2015}).
\btitle{Large deviations for random projections of $\ell^p$ balls}.
\bjournal{Preprint}.
\end{barticle}
\endbibitem

\bibitem{gantert2014large}
\begin{barticle}[author]
\bauthor{\bsnm{Gantert},~\bfnm{Nina}\binits{N.}},
  \bauthor{\bsnm{Ramanan},~\bfnm{Kavita}\binits{K.}} \AND
  \bauthor{\bsnm{Rembart},~\bfnm{Franz}\binits{F.}}
(\byear{2014}).
\btitle{Large deviations for weighted sums of stretched exponential random
  variables}.
\bjournal{Electronic Communications in Probability}
\bvolume{19}
\bpages{1--14}.
\end{barticle}
\endbibitem

\bibitem{kiesel2000large}
\begin{barticle}[author]
\bauthor{\bsnm{Kiesel},~\bfnm{R{\"u}diger}\binits{R.}} \AND
  \bauthor{\bsnm{Stadtm{\"u}ller},~\bfnm{Ulrich}\binits{U.}}
(\byear{2000}).
\btitle{A large deviation principle for weighted sums of independent
  identically distributed random variables}.
\bjournal{Journal of Mathematical Analysis and Applications}
\bvolume{251}
\bpages{929--939}.
\end{barticle}
\endbibitem

\bibitem{luenberger1968quasi}
\begin{barticle}[author]
\bauthor{\bsnm{Luenberger},~\bfnm{David~G}\binits{D.~G.}}
(\byear{1968}).
\btitle{Quasi-convex programming}.
\bjournal{SIAM Journal on Applied Mathematics}
\bvolume{16}
\bpages{1090--1095}.
\end{barticle}
\endbibitem

\bibitem{najim2002cramer}
\begin{barticle}[author]
\bauthor{\bsnm{Najim},~\bfnm{Jamal}\binits{J.}}
(\byear{2002}).
\btitle{A {C}ram{\'e}r type theorem for weighted random variables}.
\bjournal{Electronic Journal of Probability}
\bvolume{7}
\bpages{1-32}.
\end{barticle}
\endbibitem

\bibitem{rockafellar1970convex}
\begin{bbook}[author]
\bauthor{\bsnm{Rockafellar},~\bfnm{R~Tyrrell}\binits{R.~T.}}
(\byear{1970}).
\btitle{Convex Analysis}
\bvolume{28}.
\bpublisher{Princeton University Press}.
\end{bbook}
\endbibitem

\bibitem{romano1986counterexamples}
\begin{bbook}[author]
\bauthor{\bsnm{Romano},~\bfnm{Joseph~P}\binits{J.~P.}} \AND
  \bauthor{\bsnm{Siegel},~\bfnm{Andrew~F}\binits{A.~F.}}
(\byear{1986}).
\btitle{Counterexamples in Probability and Statistics}.
\bpublisher{CRC Press}.
\end{bbook}
\endbibitem

\end{thebibliography}
\bibliographystyle{imsart-number}

\end{document}